\documentclass[11pt]{amsart}
\usepackage[utf8]{inputenc}
\usepackage{graphicx,amssymb,amsmath,amsthm,bm}
\usepackage{caption}
\usepackage{subcaption}
\usepackage{indentfirst}
\usepackage{cancel}
\usepackage{lipsum}
\usepackage{epstopdf,epsfig}
\usepackage{enumerate,color} 
\usepackage{ulem}

\usepackage{array,ragged2e}
\newcolumntype{C}{>{\Centering\arraybackslash}m{0.14\linewidth}}

\usepackage{accents}

\usepackage[overload]{empheq}
\usepackage{soul}
\usepackage[multiple]{footmisc}
\numberwithin{equation}{section}

\theoremstyle{plain}
\newtheorem{theorem}{Theorem}[section]
\newtheorem{lemma}[theorem]{Lemma}

\theoremstyle{definition}

\newtheorem*{defi*}{Definition} 
\theoremstyle{remark}
\newtheorem{remark}{Remark}

\theoremstyle{remark}

\usepackage{etoolbox}
\makeatletter
\let\alignts@preamble\align@preamble
\patchcmd{\alignts@preamble}{\displaystyle}{\textstyle}{}{}
\patchcmd{\alignts@preamble}{\displaystyle}{\textstyle}{}{}

\def\alignts{\let\align@preamble\alignts@preamble\start@align\@ne\st@rredfalse\m@ne}

\makeatother

\allowdisplaybreaks

\title[Singularity formation]{Singularity formation of hydromagnetic waves in cold plasma}

\author[J. Bae]{Junsik Bae}
\address[JB]{Department of Mathematical Sciences, Ulsan National Institute of Science and Technology, Ulsan, 44919, Republic of Korea}
\email{junsikbae@unist.ac.kr}

\author[J. Choi]{Junho Choi}
\address[JC]{School of Mathematical Sciences, Korea Advanced Institute of Science and Technology, Daejeon, 34141, Republic of Korea}
\email{junho\_choi@kaist.ac.kr}

\author[B. Kwon]{Bongsuk Kwon}
\address[BK]{Department of Mathematical Sciences, Ulsan National Institute of Science and Technology, Ulsan, 44919, Republic of Korea}
\email{bkwon@unist.ac.kr}

\date{\today}

\begin{document}
 
\maketitle 

\begin{abstract}
We study $C^1$ blow-up of the compressible fluid model introduced by Gardner and Morikawa, which describes the dynamics of a magnetized cold plasma.  We propose sufficient conditions that lead to  $C^1$ blow-up. In particular, we find that smooth solutions can break down in  finite time even if the gradient of initial velocity is identically zero. The density and the gradient of the velocity  become unbounded as time approaches the lifespan of the smooth solution. The Lagrangian formulation reduces the singularity formation problem to finding a zero of the associated second-order ODE.\\


\end{abstract} 

\section{Introduction} 
Under suitable assumptions, the motion of a magnetized  cold plasma can be described by the the following simplified model  \cite{Gar}:
\begin{subequations}\label{GM2}
\begin{align}
& \rho_t +  (\rho u)_x = 0, \label{GM2_1} \\ 
& u_t  + u u_x = vB, \label{GM2_2} \\
& B_x = -\rho v, \label{GM2_3} \\
& v_x = \rho - B,  \label{GM2_4}
\end{align}
\end{subequations} 
where $\rho>0$ represents the number density of ions, $u$ is the $x$-component of the ion velocity, $v$ is the difference between the $y$-components of the ion and electron velocities, and $B$ is the $z$-component of the magnetic field. All unknowns in \eqref{GM2} are functions of  $(t,x)\in [0,\infty)\times \mathbb{R}$. The system \eqref{GM2} can be formally derived from the two-species 3D Euler-Maxwell system \cite{Ch}  under the following assumptions: (i) all unknown functions are uniform in the $y$ and $z$ directions, (ii) 
the magnetic field is applied only in the $z$ direction, (iii)   the densities of ions and electrons are the same (quasineutrality), (iv) slow motion (the displacement current is neglected), (v)  cold plasma (the pressure effects are neglected), and (vi) \eqref{GM2_4} holds at $t=0$. 

The system \eqref{GM2} was introduced in \cite{Gar} to investigate  hydromagnetic waves propagating across a magnetic field. Despite being one of the first examples from which the KdV equation was derived outside the context of water waves, the system \eqref{GM2} has not received much attention. We refer to \cite{BeKa,KOTW} for studies on the oblique propagation of hydromagnetic waves. In \cite{PuLi}, the KdV limit of \eqref{GM2} is rigorously justified.  The work of  \cite{ADG} formally derives some asymptotic models of \eqref{GM2} and investigates their properties. In particular, wave-breaking phenomena (derivative blow-up) can occur in these asymptotic models of \eqref{GM2}; see \cite{ADG,YC}.

 We show that the solution to \eqref{GM2}  blows up in finite time for a certain class of initial data. Our result is particularly interesting since it implies that the solution may blow up even if the gradient of the initial velocity $u_0$ is identically zero.

\subsection{Main result}
We consider smooth solutions to \eqref{GM2} with the far-field state $(\rho,u,v,B) \to (1,0,0,1)$ as $|x| \to \infty$. 
For initial data $(\rho_0-1,u_0)\in H^2(\mathbb{R})\times H^3(\mathbb{R})$, the classical  solution to the system \eqref{GM2} exists locally in time \cite{AlGr}. For given $\rho>0$, we see that  $v$ and $B$ are determined by \eqref{GM2_3}--\eqref{GM2_4}: $v_{xx} - \rho v = \rho_x$ and $B_x=-\rho v$. As long as the smooth solution to \eqref{GM2} exists, the energy
\begin{equation}\label{energy}
H(t):= \frac{1}{2}\int_{\mathbb{R}} \rho(u^2+v^2) + (B-1)^2  \,dx
\end{equation} 
is conserved, i.e., $H(0)=H(t)$ for $t \geq 0$.

Now, we present our main theorem. Let $h_0:= 2H(0) + \sqrt{4(H(0))^2 + 2H(0)}$.
\begin{theorem}\label{MainThm}
If the initial data satisfies one of the following:
\begin{subequations}\label{Blow}
\begin{align}
& (i) \quad   h_0 < 1 \quad \text{and}\quad   \rho_0(\alpha) < \frac{(1-h_0)^2}{2(1+h_0)} \quad \text{for some } \alpha \in\mathbb{R}, \label{MainThm1} \\ 
& (ii) \quad  h_0 < 1 \quad \text{and}\quad  -\sqrt{2(1+h_0)\rho_0(\alpha) - (1-h_0)^2} \geq \partial_x u_0(\alpha) \quad \text{for some } \alpha \in\mathbb{R}, \label{ThmC1}  \\
& (iii) \quad  -\sqrt{2(1+h_0)\rho_0(\alpha) } \geq \partial_x u_0(\alpha) \quad \text{for some } \alpha \in\mathbb{R}, \label{ThmC5}
\end{align}
\end{subequations}  
then  the maximal existence time $T_\ast$ of the classical solution to the system \eqref{GM2} is finite. Moreover,   $\| \rho(t,\cdot) \|_{L^\infty} \to \infty$ and $\| \partial_x u(t,\cdot) \|_{L^\infty} \to \infty$ as $t\nearrow T_\ast$.
\end{theorem}

We notice that the condition \eqref{MainThm1} does not require $u_0$ to have a negative gradient. We illustrate a class of the initial data satisfying $h_0<1$ and \eqref{MainThm1}. Using the identity \eqref{Iden1}, we obtain
\[
H(0) \leq \frac{\sup_{x\in \mathbb{R}}\rho_0}{2}\int_{\mathbb{R}}|u_0|^2\,dx + \int_{\mathbb{R}} (  \rho_0 -1)^2\,dx. 
\]

Hence, it is clear that one can choose the initial data such that $H(0) \ll 1$ and $\inf \rho_0 < 1/2$. In particular, one can choose $u_0 \equiv 0$ and $\rho_0(\alpha) = 1 - \delta \eta_\mu(\alpha)$, where $\eta_\mu$ is a standard bump function supported on $[-\mu, \mu]$, with $\mu>0$ and $\delta>0$ being sufficiently small.

\begin{figure}[h]
\begin{center}
\includegraphics[width=.92\linewidth]{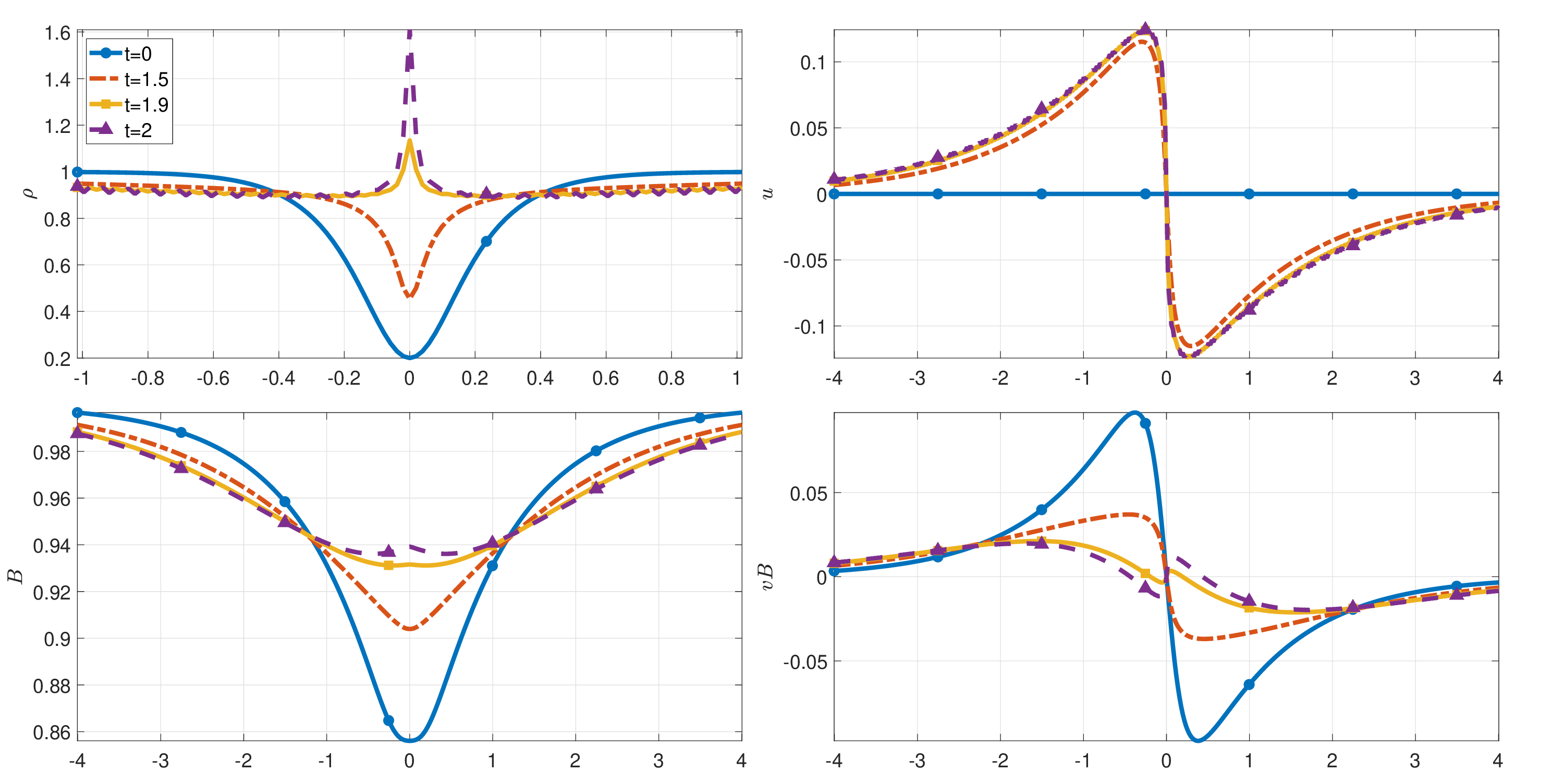}
\end{center}
\caption[Caption for LOF]{A numerical example to demonstrate a solution-blow-up scenario which holds \eqref{MainThm1}. When setting $\rho_0=1-0.8sech(7x)$, $u_0=0$ as initial conditions for \eqref{GM2}, the energy \eqref{energy} is found to be $H(0)=0.0247$. This leads to $h_0=0.2774<1$ which satisfies the criterion \eqref{MainThm1} at $\alpha=0$, that is, $0.2=\rho(0)<\frac{(1-h_0)^2}{2(1+h_0)}=0.2044$. Top-left: $\rho$ profiles, top-right: $u$, bottom-left: $B$, and bottom-right: $vB$ at $t=0$, $t=1.5$, $t=1.9$, $t=2.0$.  }
\label{GM_figures}

\end{figure} 

\begin{figure}[h]
\begin{center}

\includegraphics[width=.92\linewidth]{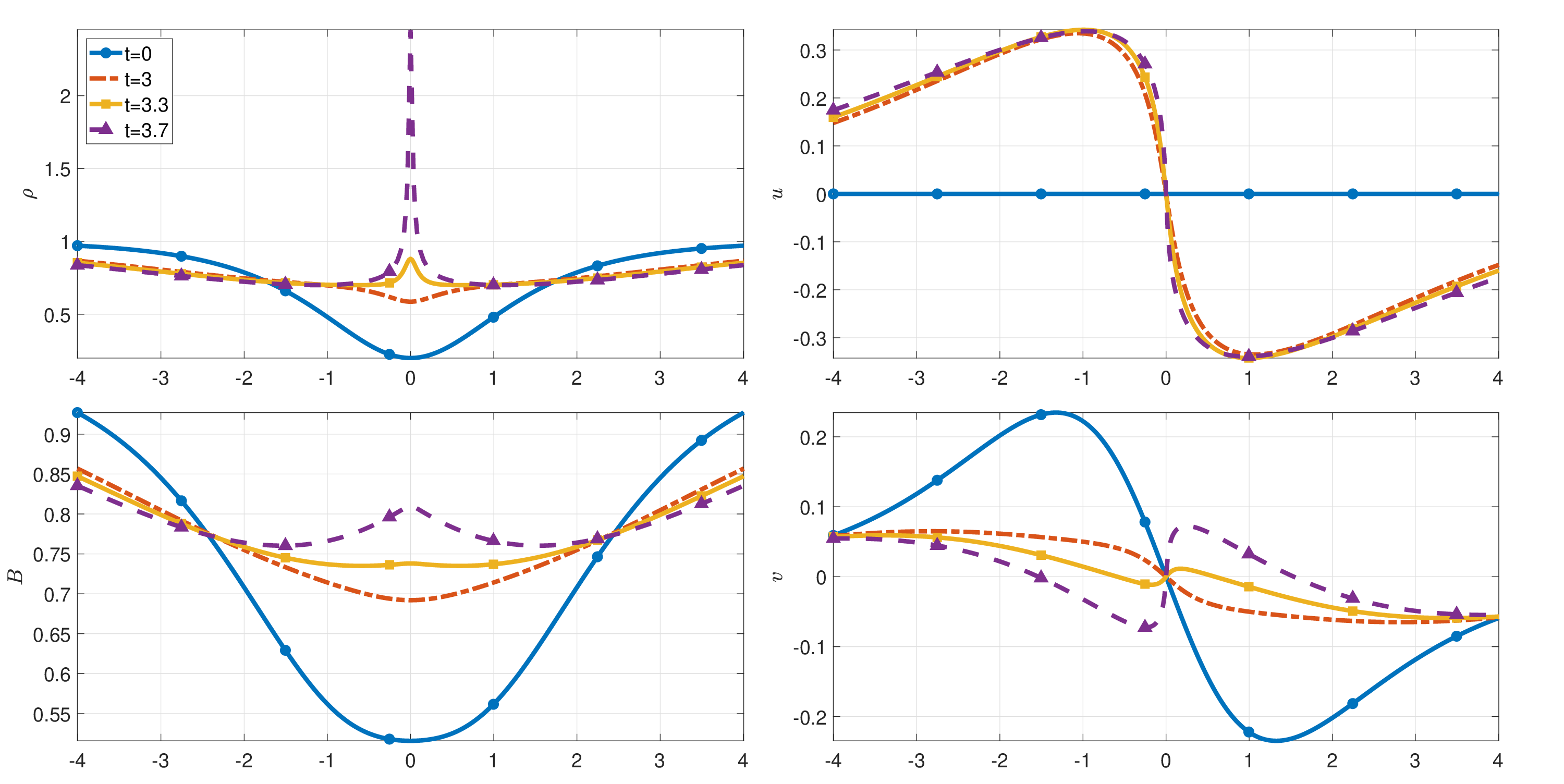}

\end{center}

\caption[Caption for LOF]{A numerical example to demonstrate a solution-blow-up scenario which does not hold \eqref{MainThm1}. When setting $\rho_0=1-0.8sech(x)$, $u_0=0$ as initial conditions for \eqref{GM2}, the energy \eqref{energy} is found to be $H(0)=0.4977$. This leads to $h_0=2.4049$ which is not less than $1$ in the criterion \eqref{MainThm1}. However, it is computed that $0.2=\rho(0)<\frac{(1-h_0)^2}{2(1+h_0)}=0.2897$ at $\alpha=0$. Top-left: $\rho$ profiles, top-right: $u$, bottom-left: $B$, and bottom-right: $v$ at $t=0$, $t=3$, $t=3.3$, $t=3.7$.  }
\label{GM_figures1}
\end{figure}

   As $t$ approaches the blow-up time $T_\ast$, $\|v_x(t,\cdot)\|_{L^\infty}$ diverges  since $\|B(t,\cdot)\|_{L^\infty}$ remains  uniformly bounded on $[0,T_\ast)$ (see \eqref{GM2_4} and Lemma \ref{thm2}). On the other hand, our numerical simulations show that $\|v(t,\cdot)\|_{L^\infty}$ is also uniformly bounded on $[0,T_\ast)$  (see Figure \ref{GM_figures} and \ref{GM_figures1}) resulting in the blow-up of $B_x$ due to \eqref{GM2_3}. We also remark that various simulations suggest that sufficient conditions \eqref{Blow} are not optimal (see Figure \ref{GM_figures1}).

The gradient of the velocity blows up due to the hyperbolic part \eqref{GM2_2} of \eqref{GM2}. On the other hand, due to the absence of pressure, the system \eqref{GM2} is \textit{weakly coupled} and not hyperbolic, leading to the density blow-up. A natural question arises: what is the asymptotic behavior of solutions near (and at) the blow-up time and location? Specifically, at the blow-up time, (1) whether the density blow-up profile is the so-called delta shock, and (2) whether the blow-up profile for $u$ exhibits a jump discontinuity. In fact, the same question was posed in \cite{BCK} for the pressureless Euler-Poisson system, and it was  shown  in \cite{BKK} that, generically, the density is not a Dirac measure and the velocity exhibits $C^{1/3}$   regularity at the blow-up time. It would be interesting to investigate more precise structure of singularities in the solutions to \eqref{GM2}.

To prove Theorem \ref{MainThm}, we follow the strategy of \cite{BCK} (see also Appendix 5.5 in \cite{BKK}), which is two-fold:  we   derive a second-order ODE using the Lagrangian formulation and establish a uniform (in $x$ and $t$) bound of $B$. More specifically, by defining
$
w(t,\alpha):= \frac{\partial x}{\partial \alpha}(t,\alpha),
$
where  $x(t,\alpha)$ is the characteristic curve satisfying $\frac{dx}{dt}(t,\alpha) = u(t,x(t,\alpha))$, $x(0,\alpha)=\alpha \in \mathbb{R}$, we derive the initial value problem for the second-order ODE  for $w$:
\begin{equation}\label{ODEA}
\frac{d^2w}{dt^2} (t,\alpha) +B^2w(t,\alpha) = B\rho_0(\alpha)  - v^2\rho_0(\alpha), \quad w(0,\alpha) =1, \quad \frac{dw}{dt}(0,\alpha) = \partial_x u_{0}(\alpha),
\end{equation}
where $B$ and $v$ are evaluated at $(t,x(t,\alpha))$. If $w$ vanishes at some finite time $t=T_\ast$, then the solution to \eqref{GM2} blows up in the $C^1$ topology. Hence, it boils down to finding sufficient conditions that guarantee $w$ vanishes in finite time. 

We note that \eqref{ODEA} is not a closed ODE since $B$ and $v$ are not given functions but are determined by $\rho$ via \eqref{GM2_3}--\eqref{GM2_4}. Nevertheless, by obtaining the uniform bound for $B$ in Lemma \ref{thm2},  we can perform  a comparison using the associated second-order differential inequality (see \eqref{2ndODI}).

%

The numerical simulations presented in Figures \ref{GM_figures} and \ref{GM_figures1} are conducted using the implicit pseudo-spectral method as detailed in \cite{LS} with $\Delta x=10/2^{10}$ in the spatial domain $[-10,10]$ and the Crank-Nicolson method with $\Delta t=0.01$ in the temporal domains $[0,2]$ and $[0,3.7]$, respectively.
 
\section{Proof of Theorem \ref{MainThm}}
We first show the key lemma concerning with the uniform estimates for $B$. 
\begin{lemma}\label{thm2}
As long as the classical solution to \eqref{GM2} exists, it holds that for $t\geq 0$,
\begin{subequations} 
\begin{align}
&  \sup_{x \in \mathbb{R}}| B -1 |   \leq 2 H(0) + \sqrt{4(H(0))^2+2 H(0)} =: h_0, \label{GM_B_bound_Lem} \\ 
&  \int_{\mathbb{R}} (\rho-1)^2 \,dx = \int_{\mathbb{R}} (B-1)^2 + |v_x|^2  +2\rho v^2 \,dx. \label{Iden1}
\end{align}
\end{subequations} 
\end{lemma} 

\begin{proof} In this proof, we let $\tilde{B}:=B-1$ and $\tilde{\rho}:=\rho-1$ for simplicity.
We first show \eqref{GM_B_bound_Lem}. Using Young's inequality and \eqref{GM2_3}, we have
\begin{equation}\label{GM_B_bound}
\begin{split}
\frac{1}{2}|\tilde{B}|^2 
& = \int_{-\infty}^x \tilde{B}\tilde{B}_y\,dy   \leq \int_\mathbb{R}  \frac{\tilde{B}_x^2}{2\rho} + \frac{\rho\tilde{B}^2}{2} \,dx  
 = \int_\mathbb{R} \frac{\rho v^2}{2} + \frac{\tilde{B}^2}{2} + \frac{\tilde{\rho}\tilde{B}^2}{2}\,dx.
\end{split}
\end{equation}
 Using \eqref{GM2_4}, integrating by parts, and then using \eqref{GM2_3}, we obtain
\begin{equation*}
\begin{split}
\int_\mathbb{R} (\tilde{\rho} - \tilde{B})\tilde{B}^2 \,dx 
&  = \int_\mathbb{R} v_x \tilde{B}^2 \,dx  
 = -\int_\mathbb{R} 2v\tilde{B}\tilde{B}_x \,dx
  = \int_\mathbb{R} 2\rho v^2\tilde{B}\,dx,
\end{split}
\end{equation*}
which implies 
\begin{equation}\label{GM_B_bound2}
\begin{split}
\int_\mathbb{R} \frac{\tilde{\rho}\tilde{B}^2}{2}\,dx 
& = \int_\mathbb{R} \left(\rho v^2 + \frac{\tilde{B}^2}{2}\right)\tilde{B}\,dx  
 \leq \sup_{x \in \mathbb{R}}|\tilde{B}|\int_\mathbb{R} \rho v^2 + \frac{\tilde{B}^2}{2}\,dx
  \leq 2H(0)\sup_{x \in \mathbb{R}}|\tilde{B}|.
\end{split}
\end{equation}
Combining \eqref{GM_B_bound} and \eqref{GM_B_bound2}, we get
\begin{equation}
\frac{1}{2}\left(\sup_{x \in \mathbb{R}}|\tilde{B}|\right)^2 \leq \int_\mathbb{R} \frac{\rho v^2}{2} + \frac{\tilde{B}^2}{2} \,dx + 2H(0)\sup_{x \in \mathbb{R}}|\tilde{B}| \leq H(0) + 2H(0)\sup_{x \in \mathbb{R}}|\tilde{B}|,
\end{equation}
which yields \eqref{GM_B_bound_Lem}.

Now we show \eqref{Iden1}. Using  \eqref{GM2_4}, integrating by parts, and then using \eqref{GM2_3}, we get
\[
\begin{split}
\int_{\mathbb{R}} (\rho-1)^2 \,dx
& = \int_{\mathbb{R}} (B-1+v_x)^2 \,dx = \int_{\mathbb{R}} \tilde{B}^2 + |v_x|^2 - 2\tilde{B}_x v \,dx  = \int_{\mathbb{R}} \tilde{B}^2 + |v_x|^2  +2\rho v^2 \,dx.
\end{split}
\]
 This finishes the proof.  
\end{proof}

Now we derive the second-order ODE \eqref{ODEA}.  For $u\in C^1$, let $x(t,\alpha)$ be the solution to the ODE
\begin{equation}\label{ODE1}
x'(t,\alpha) = u(t,x(t,\alpha)), \quad x(0,\alpha)=\alpha \in \mathbb{R}, \quad t \geq 0,
\end{equation}
where $':=d/dt$. Here, we consider the initial position $\alpha$ as a parameter.  By taking $\partial_\alpha$ of \eqref{ODE1}, we have
\begin{equation}\label{ODE2}
w'(t,\alpha) = u_x(t,x(t,\alpha))w(t,\alpha), \quad w(0,\alpha)=1,
\end{equation}  
where $w(t,\alpha):=\partial_\alpha x(t,\alpha)$. Integrating \eqref{ODE2}, we get
\begin{equation}
w(t,\alpha) = \exp\left(\textstyle{\int_0^t u_x(s,x(\alpha,s))\,ds}\right).
\end{equation}
 Using \eqref{GM2_1} and \eqref{ODE2}, one can see that  
\begin{equation}\label{ODE4}
w(t,\alpha)\rho(t,x(t,\alpha)) =\rho_0(\alpha).
\end{equation}
By taking $\partial_x$ of \eqref{GM2_2}, we have
\begin{equation}\label{ODE3}
u_x' + (u_x)^2 = (vB)_x = v_xB + vB_x = (\rho-B)B -\rho v^2
\end{equation}
where we have used \eqref{GM2_3} and \eqref{GM2_4}.  Using \eqref{ODE2}, \eqref{ODE3} and \eqref{ODE4}, we get \eqref{ODEA}. Indeed, we have
\begin{equation*} 
\begin{split}
w'' 
& = (u_x w)'  \\
&  = u_x^2w + (-u_x^2+(\rho-B)B -\rho v^2)w  \\
&  = ((\rho-B)B -\rho v^2)w  \\
&   = B\rho_0 - B^2w - v^2\rho_0.
\end{split}
\end{equation*}
From \eqref{ODE4}, we see that $w>0$ since $\rho>0$ and that 
\begin{equation}\label{A1}
\lim_{t \to T_\ast} w(t,\alpha) = 0 \quad \Leftrightarrow \quad \rho(t,x(t,\alpha)) = \infty \quad \Rightarrow  \quad \liminf_{t \to T_\ast} u_x(t,x(t,\alpha)) = -\infty.
\end{equation}

\begin{remark}\label{Rem1}
If $v$ is (uniformly) bounded as long as the solution exists (or if we have  $\lim_{t\to T_\ast}w'(t)<\infty$), one can  show that $\liminf_{t \to T_\ast} u_x(t,x(t,\alpha)) = -\infty$ implies  $\rho(t,x(t,\alpha)) = \infty $. Furthermore, one can also  obtain the blow-up rate $u_x(t,x(t,\alpha)) \sim (t-T_\ast)^{-1}$. We refer to the proof of Lemma 2.3 of \cite{BCK}.
\end{remark} 

In what follows, we prove Theorem \ref{MainThm}.  We first consider the case (i). For the initial data satisfying \eqref{MainThm1}, we let $T_\ast$ be the maximal existence time of the $C^1$ solution to \eqref{GM2}. Then, using  \eqref{GM_B_bound_Lem} for \eqref{ODEA}, we have  
\begin{equation}\label{2ndODI}
w'' + (1-h_0)^2w  \leq (1+h_0)\rho_0(\alpha).
\end{equation} 
We let $a := (1-h_0)^2$ and  $b:= (1+h_0)\rho_0(\alpha)$.  Then, \eqref{2ndODI} can be put in the form of $w'' + aw-b =:f \leq 0$, and we have (recalling $w(0,\alpha)=1$)
\begin{equation}\label{2ndODII}
w(t) = \left( 1-\frac{b}{a} \right) \cos(\sqrt{a} t) + w'(0) \frac{\sin(\sqrt{a} t)}{\sqrt{a}} + \frac{b}{a} + \frac{1}{\sqrt{a}} \int_0^t \sin(\sqrt{a} t') f(t-t')\,dt',
\end{equation}
where we let $w(t)=w(t,\alpha)$ for notational simplicity. 
Since $2b/a \leq 1$ (i.e., \eqref{MainThm1} holds) and the integrand in \eqref{2ndODII} is nonpositive on $t\in  [0,\pi/\sqrt{a}]$, by putting $t=\pi/\sqrt{a}$ into \eqref{2ndODII}, we have $w(\pi/\sqrt{a}) \leq -1 + 2b/a  \leq 0$.  Hence, we conclude that $w$ must vanish at some point  on the interval $[0,\pi/\sqrt{a}]$ by the intermediate value theorem. From \eqref{A1}, this implies that the maximal existence time $T_\ast$ is finite. 

Next, we consider the case (ii). Consider the $C^1$ solution to \eqref{GM2} with the initial data satisfying  \eqref{ThmC1}. Using the trigonometric identity, \eqref{2ndODII} becomes
\begin{equation}\label{2ndODII1}
w(t) = \sqrt{\frac{(w'(0))^2}{a} + \left(1-\frac{b}{a}\right)^2}\cos(\sqrt{a}t  - \theta) + \frac{b}{a} + \frac{1}{\sqrt{a}} \int_0^t \sin(\sqrt{a} t') f(t-t')\,dt',
\end{equation}
where  
\[
\sin \theta = \frac{w'(0)}{\sqrt{a}\sqrt{\frac{(w'(0))^2}{a} + (1-b/a)^2}}, \quad \cos \theta = \frac{1-b/a}{\sqrt{\frac{(w'(0))^2}{a} + (1-b/a)^2}}.
\]
Since $w'(0) \leq 0$,  $\theta\in[-\pi,0]$. Hence, $\sqrt{a}t-\theta \in [0,2\pi]$ for $\sqrt{a}t \in [0,\pi]$, and the cosine function has the minimum value $-1$. We recall that the integrand of \eqref{2ndODII1} is nonpositive on $t\in[0,\pi/\sqrt{a}]$. Hence, from  \eqref{2ndODII1}, we see that if $2b \geq a$ and $w'(0) \leq 0$ (i.e., \eqref{ThmC1} holds), then  $w$ vanishes at some point on $[0,\pi/\sqrt{a}]$.

Lastly, we consider the case (iii). From \eqref{ODEA}, we obtain that $ w'' \leq (1+h_0)\rho_0$ since $w\geq 0$ and $\rho_0>0$.   Integrating it twice in $t$, we have 
\[
w(t) \leq \frac{1}{2}(1+h_0)\rho_0\left( t + \frac{w'(0)}{(1+h_0)\rho_0} \right)^2 - \frac{(w'(0))^2}{2(1+h_0)\rho_0} + 1.
\]
Hence, if \eqref{ThmC5} holds, then $w$ must vanish in  finite time. We finish the proof of Theorem \ref{MainThm}.

\section*{Acknowledgments.}
J.B. was supported by the National Research Foundation of Korea grant funded by the Ministry of Science and ICT (2022R1C1C2005658). J.C. was supported by the National Research Foundation of Korea grant funded by the Korean government (MSIT) (2021R1C1C2008763). B.K. was supported by Basic Science Research Program through the National Research Foundation of Korea (NRF) funded by the Ministry of science, ICT and future planning (NRF-2020R1A2C1A01009184).


 \end{document}